\documentclass[12pt]{amsart}
\usepackage{amscd}     
\usepackage{amssymb}
\usepackage{amsmath, amsthm, graphics}
\usepackage{xypic}     
\LaTeXdiagrams        
\usepackage[all]{xy}
\xyoption{2cell} \UseAllTwocells \xyoption{frame} \CompileMatrices
\allowdisplaybreaks[3]
\usepackage{amsfonts}
\usepackage[normalem]{ulem}
\usepackage{hyperref}
\usepackage{tikz}
\usepackage{mathtools}
\usepackage{verbatim} 

\usepackage[margin=4cm]{geometry}
\usepackage{bbm}

\usepackage{latexsym}
\usepackage{epsfig}
\usepackage{amsfonts}
\usepackage{enumerate}
\usepackage{times,float}
\usepackage{graphics}

\newtheorem{prop}{Proposition}
\newtheorem{lem}[prop]{Lemma}

\newtheorem{cor}[prop]{Corollary}
\newtheorem{thm}[prop]{Theorem}

\newtheorem{defn}[prop]{Definition}

\theoremstyle{remark}

\theoremstyle{remark}

\numberwithin{equation}{section}

\newcommand{\Mbar}{\overline{\M}}

\newcommand{\M}{\mathcal{M}}

\newcommand{\ev}{{\mathrm{ev}}}

\newcommand{\Hilb}{\mathsf{Hilb}}

\newcommand{\BC}{{\mathbb{C}}}
\newcommand{\BL}{{\mathbb{L}}}
\newcommand{\BQ}{{\mathbb{Q}}}

\def\b1{{\mathbf 1}}

\begin{document}

\title[Universality of $\mathsf{CohFT}$ of $\Hilb_n(S)$]{Universality of $\mathsf{CohFT}$ of $\Hilb_n(S)$}

\author[Denis Nesterov]{Denis Nesterov}
\address{Department of  Mathematics\\ ETH Z\"urich \\ R\"amistrasse 101 \\ 8092 Z\"urich\\ Switzerland}
\email{denis.nesterov@math.ethz.ch}

\author[Hsian-Hua Tseng]{Hsian-Hua Tseng}
\address{Department of Mathematics\\ Ohio State University\\ 231 West 18th Ave.\\ Columbus,  OH 43210\\ USA}
\email{hhtseng@math.ohio-state.edu}

\begin{abstract}
By virtue of their relation to Fulton--MacPherson compactifications, the cohomological field theory of Hilbert schemes of points on a surface, in extremal curve classes, is universally and effectively determined by the characteristic numbers of the surface and by the corresponding theory of Hilbert schemes of the affine plane. As a consequence, we compute the extremal genus 1 Gromov--Witten invariants of Hilbert schemes of all surfaces. 
\end{abstract}

\maketitle

\section{Introduction}
\subsection{Overview}
Let $S$ be a smooth complex quasi-projective surface. For $n\in \mathbb{N}$, denote by
$$\Hilb_n(S)$$
the Hilbert scheme of $n$ points on $S$. In \cite{N1}, using wall-crossing methods, the Gromov--Witten theory of $\Hilb_n(S)$ in extremal\footnote{Also known as exceptional curve classes.} curve classes is shown to be effectively determined by the Gromov--Witten theory of $\Hilb_n(\BC^2)$ and integrals on Fulton--MacPherson compactifications $\mathsf{FM}_n(S)$. It therefore depends universally on the characteristic numbers of the surface $S$, as stated in Corollary~\ref{main}, generalizing the universality of the classical intersection theory established by Ellingsrud, G\"ottsche, and Lehn \cite{EGL}. This holds at the level of Gromov--Witten classes in $H^*(\overline{\mathcal{M}}_{g,k})$, or, in other words, cohomological field theory.

We use the results of \cite{N1}, in particular Theorem~\ref{wall}, together with the calculations of \cite{ILPT} to compute the genus 1 Gromov--Witten invariants of $\Hilb_n(S)$ for all torus-equivariant quasi-projective surfaces with projective torus-fixed loci.

\begin{thm} \label{genus1} 
Let $D=\mathrm{c}_1(\mathcal{O}^{[n]})$ be the boundary divisor of $\Hilb_n(S)$. Consider the generating series\footnote{Here the $n=0,1$ terms vanish because $D=0$ in these cases.} of the genus 1 Gromov--Witten invariants in extremal curve classes,  
\begin{equation}\label{eqn:g1_gen_func}
 F(Q,q):=\sum_{n\geq 0} \Big( \sum_{d\geq 0} \langle D \rangle_{1,d}^{\Hilb_n(S)} q^d\Big) Q^n.
\end{equation}
The series $F(Q,q)$ admits the following expression,
\[
F(Q,q)=-\frac{1}{24} P(Q)^{\int_S\mathrm{c}_2(S)-1}  T(Q,q)\Big( \int_S \mathrm{c}_1(S)^2+\Big(\int_S\mathrm{c}_2(S)\mathrm{c}_1(S)\cdot \int_S\mathrm{c}_1(S)\Big) \log P(Q)\Big), 
\]
where 
\begin{equation}\label{eqn:P_and_T} 
P(Q):=\sum_{n\geq 0} p(n)Q^n, \quad 
T(Q,q):=\sum_{n\geq 0} \Big( \sum_{d\geq 0}\mathrm{Tr}_{n,d} q^d   \Big) Q^n,
\end{equation}
such that $\mathrm{Tr}_{n,d}$ is the trace of the quantum multiplication by the divisor $D$ on $\Hilb_n(\BC^2)$, and $p(n)$ is the number of partitions of $n$. 
\end{thm}

For a projective surface $S$ without a torus action, the formula above admits a simplification, since the integrals 
\[
\int_S\mathrm{c}_2(S)\mathrm{c}_1(S) \cdot \int_S\mathrm{c}_1(S)
\]
vanish. 
\begin{cor} 
Let $S$ be a projective surface without a torus action, then 
\[ 
F(Q,q)=-\frac{1}{24} \int_S \mathrm{c}_1(S)^2 \, P(Q)^{\int_S\mathrm{c}_2(S)-1} \, T(Q,q).
\]
\end{cor}	

This agrees with the $n=2,3$ computations in \cite{HQ} and solves \cite[Conjecture 1.3]{HQ} in an explicit form. 

\subsection{Classical intersection numbers}
A {\em universality principle} asserts that tautological intersection numbers on $\Hilb_n(S)$ are given by expressions depending on $S$ only through characteristic numbers. A basic example are the intersection numbers of the form
\begin{equation}\label{eqn:classical_integrals}
\int_{\Hilb_n(S)}P(\{c_i(L^{[n]})\}, \{c_j(T_{\Hilb_n(S)})\}),    
\end{equation}
where $P$ is a homogeneous polynomial of degree $2n$, and 
\begin{equation}
L^{[n]}:=(p_2)_*(p_1^*L\otimes \mathcal{O}_{\mathcal{Z}})    
\end{equation}
is the rank $n$ tautological vector bundle constructed from a line bundle $L\to S$ using the universal subscheme 
\begin{equation*}
\xymatrix{
\mathcal{Z}\ar@{^{(}->}[r] & S\times \Hilb_n(S) \ar[r]^(0.6){p_2}\ar[d]^{p_1} & \Hilb_n(S)\\
\, & S. & \,
}
\end{equation*}
By \cite{EGL}, assuming $S$ is projective and non-equivariant, (\ref{eqn:classical_integrals}) are given by expressions that depend on $S$ and $L$ only through the characteristic numbers 
\[
\int_S \mathrm{c}_1(L)^2, \quad \int_S \mathrm{c}_1(L)\mathrm{c}_1(S), \quad \int_S\mathrm{c}_1(S)^2, \quad \text{and} \quad \int_S\mathrm{c}_2(S). 
\]
Explicit examples include Euler characteristics \cite{G},
\begin{equation*}
\sum_{n\geq 0}e(\Hilb_n(S))Q^n=\frac{1}{\prod_{m\geq 1}(1-Q^m)^{\int_S\mathrm{c}_2(S)}},    
\end{equation*}
and Segre integrals \cite{Lehn}, \cite{V}, \cite{MOP1}, \cite{MOP2}.

\subsection{Extremal Gromov-Witten theory}
In this work, we are interested in the universality of extremal Gromov-Witten cohomological field theory of $\Hilb_n(S)$. The curve classes contracted by the Hilbert-Chow morphism
\begin{equation}
\rho^{\text{HC}}: \Hilb_n(S)\to \text{Sym}_n(S)    
\end{equation}
is generated by an effective curve class $\beta_n\in H_2(\Hilb_n(S),\mathbb{Z})$,
\begin{equation}
\mathbb{Z}\cdot \beta_n=\text{ker}(\rho^{\text{HC}}_*)\subseteq H_2(\Hilb_n(S),\mathbb{Z}).    
\end{equation}
The moduli space $\overline{\mathcal{M}}_{g,k}(\Hilb_n(S),d)$ of $k$-pointed, genus $g$, degree $d\beta_n$ stable maps to $\Hilb_n(S)$ has virtual dimension
\begin{equation*}
(1-g)\cdot (2n-3)+\int_{d\beta_n}\mathrm{c}_1(T_{\Hilb_n(S)})+k=(1-g)(2n-3)+k.
\end{equation*}
Extremal Gromov-Witten invariants of $\Hilb_n(S)$ are of the following form
\begin{equation*}\label{eqn:ext_GW_inv}
\langle \gamma_1,...,\gamma_k\rangle_{g,d}^{\Hilb_n(S)}:=\int_{[\overline{\mathcal{M}}_{g,k}(\Hilb_n(S),d)]^{\text{vir}}}\prod_{i=1}^k\ev_i^*\gamma_i,\quad \gamma_1,...,\gamma_k\in H^*(\Hilb_n(S)).    
\end{equation*}
More generally, the extremal cohomological field theory is assembled from the Gromov--Witten classes
\begin{equation}\label{eqn:ext_coh}
\mathrm{GW}(\gamma_1,...,\gamma_k)_{g,d}^{\Hilb_n(S)}:=q_*\Big(\prod_{i=1}^k\ev_i^*\gamma_i \cap  [\overline{\mathcal{M}}_{g,k}(\Hilb_n(S),d)]^{\text{vir}}\Big) \in H^*(\Mbar_{g,n}),
\end{equation}
where
\begin{equation*}
q \colon  \overline{\mathcal{M}}_{g,k}(\Hilb_n(S),d) \rightarrow  \overline{\mathcal{M}}_{g,k}
\end{equation*}
is the natural morphism that forgets the maps.

The theory extends to a quasi-projective surface $S$ equipped with an action of a torus $A$. In this case, $\gamma_i \in H_A^*(\Hilb_n(S))$ are $A$-equivariant classes, while the resulting invariants are defined via torus localization and take values in the localized equivariant cohomology ring $H^*(\Mbar_{g,n})[a_1^{\pm},\ldots,a_m^{\pm}]$, with $H_A^*(\text{pt})=\mathbb{Z}[a_1,...,a_m]$. In fact, the equivariant Gromov--Witten theory of $\Hilb_n(\BC^2)$ is fundamental for determining even the non-equivariant theory of $\Hilb_n(S)$ for a projective surface $S$, as stated in Theorem~\ref{wall}.

When $S$ is projective and non-equivariant, the dimension constraint implies that (\ref{eqn:ext_coh}) is possibly nonzero only when $g=0,1$. The classical intersection theory is recovered by the degree 0 invariants $\langle \gamma_1,\gamma_2,\gamma_3\rangle_{0,0}^{\Hilb_n(S)}$.

%We also consider the cases when $S$ is quasiprojective and equipped with a suitable torus action. Examples of this are $\mathbb{C}^2$ and minimal resolutions of $ADE$ surface singularities.

\subsection{Previous work}
\subsubsection{Quasi-projective case}
The work \cite{op} is fundamental in the study of Gromov--Witten theory of Hilbert schemes of points on smooth surfaces. Since $H_2(\Hilb_n(\mathbb{C}^2), \mathbb{Z})=\text{ker}(\rho^{\text{HC}}_*)$, all Gromov--Witten invariants of $\Hilb_n(\mathbb{C}^2)$ are extremal. In \cite{op}, the (equivariant) genus 0 Gromov--Witten theory of $\Hilb_n(\mathbb{C}^2)$ is fully and effectively solved. The key result is the explicit calculation of the operator $\mathsf{M}^{\Hilb_n(\mathbb{C}^2)}_D(q)$ of small quantum multiplication by the divisor
\begin{equation}
D=\mathrm{c}_1(\mathcal{O}^{[n]}).    
\end{equation}
It is proved in \cite{op} that all genus $0$ Gromov-Witten invariants of $\Hilb_n(\mathbb{C}^2)$ can be explicitly and effectively reconstructed from $\mathsf{M}^{\Hilb_n(\mathbb{C}^2)}_D(q)$ by WDVV equations.

The quantum multiplication by the divisor $D$ can be considered for any surface $S$. In \cite{MOb}, the operator\footnote{In fact, \cite{MOb} explicitly determines all divisor operators and beyond extremal degrees.} $\mathsf{M}_D^{\Hilb_n(S)}(q)$ is explicitly calculated for $S$ the minimal resolutions of $ADE$ surfaces singularities equipped with suitable torus actions. Analogous reconstructions of all genus $0$ Gromov-Witten invariants in this case can be achieved, assuming a nondegeneracy conjecture (see \cite[Section 6.3]{MOb}).

The genus $1$ Gromov--Witten theory of $\Hilb_n(\mathbb{C}^2)$ is studied in \cite{ILPT}. Here, the key role is played by the generating series
\begin{equation}
\langle D\rangle_1^{\Hilb_n(\mathbb{C}^2)}=\sum_{d\geq 0}\langle D\rangle_{1,d}^{\Hilb_n(\mathbb{C}^2)}q^d,    
\end{equation}
which is explicitly evaluated by calculating certain family Hodge integrals which are proven to be equivalent to Gromov--Witten theory of $\Hilb_n(\mathbb{C}^2)$ in \cite{pt}. Reconstructions of other genus $1$ Gromov--Witten invariants are also obtained.

A parallel study of genus $1$ Gromov--Witten invariants of $\Hilb_n(A_r)$ in all degrees will be available in the forthcoming work \cite{ILPTY}.

\subsubsection{Projective case}
For $S$ projective {\em and} simply connected, universality results of extremal Gromov--Witten invariants of $\Hilb_n(S)$ have been obtained by a combination of techniques including cosection localization: for genus $0$, $2$-pointed extremal invariants of $\Hilb_n(S)$, this is proven in \cite{ll}; for genus $0$, $3$-pointed extremal invariants of $\Hilb_n(S)$, this is proven in \cite{lq}; for genus $1$ extremal invariants of $\Hilb_2(S)$, this is proven in \cite{hlq}; for genus $1$ extremal invariants of $\Hilb_3(S)$, this is proven in \cite{HQ}. Universality of genus $1$ extremal invariants of $\Hilb_n(S)$ is formulated in \cite[Conjecture 1.3]{HQ}.

\subsection{Acknowledgment}
H.-H. T. thanks R. Pandharipande for helpful discussions. H.-H. T. is grateful to J. Hu and Z. Qin for discussions about their work \cite{HQ}. D. N. was supported by a Hermann-Weyl-instructorship at the Forschungsinstitut
f\"ur Mathematik at ETH Z\"urich.

\section{Universality of the extremal cohomological field theory}

\subsection{$I$-functions}\label{Iuniversal}  
Let $T=\BC^*\times \BC^*$ be the 2-torus acting on $\BC^2$ with weights $t_1$ and $t_2$. There is an induced action on the space of maps of degree $d$ to $\Hilb_n(\BC^2)$, 
%\[ \xymatrix{  \Mbar_{g,k}(\Hilb_n(\BC^2),d).} \ar@(dl,ul)[]^{T} \]
\[  T \curvearrowright \Mbar_{g,k}( \Hilb_n(\BC^2),d). \]
Consider a smooth complex quasi-projective surface $S$ with the action of another torus $A$, 
%\[ \xymatrix{ S, \ar@(dl,ul)[]^{A}}\]
\[ A \curvearrowright S, \]
such that the torus-fixed locus $S^A$ is projective. Denote its $A$-equivariant Chern roots by $\alpha_1(S)$ and $\alpha_2(S)$. All moduli spaces and cohomologies associated to $S$ will be treated $A$-equivariantly. The intersection numbers are assumed to be defined by the torus localisation if $S$ is not projective. 
\begin{defn}\label{I-function} 
Let 
\[
\{\mu^1(\gamma), \hdots, \mu^k(\gamma)\}
\]
be a collection of partitions of $n$ weighted by cohomology classes $\gamma_{j,i} \in H^*_A(S)$, 
\[ 
\mu^j (\gamma)=((\mu^j_{1},\gamma_{j,1}), \hdots, (\mu^j_{\ell(\mu^j)},\gamma_{j,\ell(\mu^j)})).
\]
We define the $I$-function associated to the Gromov--Witten theory of $\Hilb_n(S)$, 
\[ 
I_{n,g,d}\Bigg(\prod^{k}_{j=1} \mu^j(\gamma),z \Bigg) \in H^*_A(S\times \Mbar_{g,k} )[\![z^\pm]\!], 
\]
as the class
\[
t_1t_2  \cdot \prod_{j,i} \gamma_{j,i}  \cdot q_*\left( \frac{\prod^k_{j=1} \ev^*_j\mu^j \cap 
[\Mbar_{g,k}( \Hilb_n(\BC^2), d)^{T}]^\mathrm{vir}}{\mathrm{e}_{T}(N^\mathrm{vir})} \right),
\]
such that the variable substitution
\[
t_i \mapsto z+ \alpha_i(S), \quad i=1,2,
\]
is made, and $q \colon \Mbar_{g,k}( \Hilb_n(\BC^2),d) \rightarrow \Mbar_{g,k}$ is the natural projection. Here, $z$ is a formal variable, and the insertions are the Nakajima--Grojnowski classes\footnote{More precisely, the insertion $\mu^j$ is given by the closure of the locus of subschemes with the multiplicity profile $\mu^j$.} associated to the partitions $\mu^j$. The expression is Taylor expanded\footnote{For example, $1/(z-a)=1/z+a/z^2+a^2/z^3+\hdots$} in the range $|z|>|a|>1$, where $a$ is an element in $H^*_A(S \times \Mbar_{g,k})$.
\end{defn}

Observe that, by the properness of the fixed locus for all tori $\BC^* \subset T$ other than the coordinate tori, the Gromov--Witten classes of $\Hilb_n(\BC^2)$ are Laurent polynomials; that is, they are contained in $H^*(\Mbar_{g,k})[t_1^\pm,t_2^\pm]$. Moreover, these Laurent polynomials are symmetric in $t_i$. In particular, the Taylor expansion of an $I$-function in the above definition does not require inverting elements in $H^*_A(S \times \Mbar_{g,k})$, and the terms involving Chern roots are expressible in terms of Chern classes. 

The $I$-function is of the form
\[
C_rz^r+C_{r-1}z^{r-1}+\hdots,
\]
i.e., there are finitely many terms with non-negative powers of $z$. Since $H^*_A(S \times \Mbar_{g,k})$ is not nilpotent in general, it may contain infinitely many terms with negative powers of $z$; however, they do not play a role in Theorem \ref{wall}.

\subsection{Fulton--MacPherson compactifications}
Let $\mathsf{FM}_{n}(S)$ denote the ordered Fulton--MacPherson compactification of $n$ points, which is constructed as an iterated blow-up of $S^n$ in \cite{FM}, 
\[ 
\pi \colon \mathsf{FM}_{n}(S) \rightarrow S^n.
\]
For a subset $I \subseteq \{1,\ldots, n\}$, such that $|I| \geq 2$, the boundary divisor $D_I \in H_A^*(\mathsf{FM}_{n}(S))$ is the exceptional divisor of the blow-up along the diagonal in $S^n$ where the points labeled by $I$ coincide (or, to be more precise, its proper transform with respect to the previous blow-ups). 
Let 
\[ 
\Psi_i=\sum_{I \text{ with } i\in I}D_I, \quad i=1,\dots,n, 
\]
be the $\Psi$-classes associated to $\mathsf{FM}_{n}(S)$. By \cite{N2}, $\Psi_i$ can be expressed in terms of rank 2 vector bundles $\BL_i$ given by the cotangent spaces at the marked points, 
\[
\Psi_i=\frac{1}{2}(\mathrm{c}_1(\BL_i)+\pi_i^*\mathrm{c}_1(S)),
\]
such that $\pi_i$ is the composition of $\pi$ with the projection to the $i$-th factor of $S^n$,
\begin{equation} \label{1}
\pi_i \colon  \mathsf{FM}_{n}(S) \rightarrow S^n \rightarrow S. 
\end{equation}

To state the wall-crossing formula, Theorem \ref{wall}, we will need the moduli spaces of degree $0$ maps to $\mathsf{FM}_{n}(S)$,
\[
\mathsf{FM}_{n}(S) \times \Mbar_{g,k},
\]
equipped with the trivial obstruction theory (i.e., we will integrate over the standard fundamental class). Let 
\[ 
\pi_i 
\colon  \mathsf{FM}_{n}(S) \times \Mbar_{g,k}  \rightarrow S \times \Mbar_{g,k} 
\]
be the evaluation morphisms defined by taking products of the forgetful morphisms (\ref{1}) with the identity morphism on the second factor. 

Denote by 
\[ 
H^*_A(\text{pt})_{\text{loc}}=\BQ[a_1^\pm, \hdots, a_m^\pm]
\]
the localized $A$-equivariant cohomology of a point, where $a_i$ are the weights of $A$.  
\begin{defn}
We define the Fulton--MacPherson integrals
\[ 
\langle  \Psi^{h_1}_1 \gamma_1 \cdot \hdots \cdot \Psi^{h_n}_n \gamma_n\rangle_{g,k}^{\mathsf{FM}_n (S)} := \int_{[\mathsf{FM}_{n}(S)]} \prod^{n}_{i=1} \Psi_i^{h_i}\pi_i^* \gamma_i \in H^*(\Mbar_{g,k})[a_1^\pm, \hdots, a_m^\pm],
\]
where $\gamma_i \in H_A^*(S \times  \Mbar_{g,k})$. For the integrals on $\mathsf{FM}_{n}(S)$ (without the  $\Mbar_{g,k}$-factor), we will use the same notation but without the subscripts, 
\[ 
\langle  \Psi^{h_1}_1 \gamma_1\cdot \hdots \cdot \Psi^{h_n}_n \gamma_n\rangle^{\mathsf{FM}_n (S)} 
:= \int_{[ \mathsf{FM}_{n}(S)]} \prod^{n}_{i=1} \Psi_i^{h_i}\pi_i^* \gamma_i \in [a_1^\pm, \hdots, a_m^\pm],
\]
where $\gamma_i \in H_A^*(S)$. 
\end{defn}
Dilaton, string, and divisor equations (see \cite{N2} for the first two) give an effective way to compute the integrals of $\Psi$-classes on $\mathsf{FM}_{n}(S)$. 

\subsection{Wall-crossing} The main result of \cite{N1} implies the following expression for Gromov--Witten classes of $\Hilb_n(S)$ in extremal curves classes in terms of $I$-functions and Fulton--MacPherson integrals.
\begin{thm}[\cite{N1}]  \label{wall} 
The following relation holds in $H^*(\Mbar_{g,k})[a_1^\pm, \hdots, a_m^\pm]$, 
\[
\mathrm{GW}(\mu^1(\gamma), \hdots, \mu^k(\gamma) )^{\Hilb_n(S)}_{g,d}= \sum_{\underline{n}, \underline{d}, \underline{\mu} } \Bigg\langle  \prod^{h}_{i=1}I_{n_i,g,d_i}\Bigg(\prod^{k}_{j=1} \mu^{j,i}(\gamma),-\Psi_i\Bigg) \Bigg\rangle_{g,k}^{\mathsf{FM}_h(S)}\Big / h!,
\]
where the sum is taken over all partitions of $n$,
\[
(n_1, \hdots, n_h),
\]
partitions of $d$, 
\[ 
(d_1, \hdots, d_h),
\]
and decompositions of $\mu^j$ into subpartitions 
\[ 
\mu^j=\mu^{j,1} \sqcup \hdots \sqcup \mu^{j,h}.
\]
The parts $n_i$ are strictly positive, $d_i$ can be zero, and $|\mu^{j,i}|=n_i$. The negative powers of $\Psi_i$ are set to be $0$. 
\end{thm}

By the discussion following Definition \ref{I-function}, after substituting the $\Psi$-classes, $I$-functions have finitely many terms; that is, inserting them is well-defined. 
 
\begin{cor} \label{main} 
The Gromov--Witten classes 
\[
\mathrm{GW}(\mu^1(\gamma), \hdots, \mu^k(\gamma) )^{\Hilb_n(S)}_{g,d} \in H^*(\Mbar_{g,k})[a_1^\pm, \hdots, a_m^\pm]
\]
depend universally on a finite collection of the characteristic numbers
\[ 
\int_S \prod_{i \in I, j\in J} \gamma_{i,j} \cdot \mathrm{c}_1(S)^{k_1} \cdot \mathrm{c}_2(S)^{k_2} \in \BQ[a_1^\pm, \hdots, a_m^\pm],
\]
where $I\subseteq \{1,\hdots, k\}$ and $J\subseteq \{1,\hdots, \ell(\mu^i)\}$, and $0 \leq k_1, k_2<\infty$.

More precisely, for fixed $n,g$, $d$, and $\{\mu^1, \hdots, \mu^k\}$, there exists a $S$-indepedent polynomial with finitely many variables,
\[
P_{\mu^1, \hdots, \mu^k}(x_1,x_2, x_3,\hdots) \in H^*(\Mbar_{g,k})[x_1, x_2, x_3, \hdots],
\]
such that 
\[ 
\mathrm{GW}(\mu^1(\gamma), \hdots, \mu^k(\gamma))^{\Hilb_n(S)}_{g,d}=P_{\mu^1, \hdots, \mu^k}\Big(\int_S\mathrm{c}_1(S), \int_S\mathrm{c}_1(S)^2,\int_S\mathrm{c}_2(S), \hdots\Big)
.\]
\end{cor}
\begin{proof}
The summation in Theorem \ref{wall} is independent of $S$, so are the Gromov--Witten classes of $\Hilb_n(\BC^2)$. Hence, before integrating classes on $\mathsf{FM}_h(S)$, the formula in Theorem \ref{wall} is a $S$-independent polynomial valued in $H^*(\Mbar_{g,k})$, whose variables are the formal symbols
\begin{equation} \label{eq1}
\prod^h_{\ell=1}\pi^*_\ell\Big(\prod_{i \in I_\ell, j\in J_\ell} \gamma_{i,j} \cdot \mathrm{c}_1(S)^{k_{1,\ell}} \cdot \mathrm{c}_2(S)^{k_{2,\ell}} \Big) \cdot \Psi_\ell^{k_\ell}.
\end{equation}
 
The Fulton--MacPherson spaces $\mathsf{FM}_h(S)$ are iterated blow-ups of $S^h$ at the diagonals, such that the description of the blow-up centers is independent of $S$. The classes $\Psi_\ell$ are sums of the exceptional divisors. The normal bundles of the diagonals are box sums of the tangent bundle of $S$. Hence, the classes (\ref{eq1}) can be iteratively pushed forward to $S^h$, replacing $\Psi^{k_\ell}_\ell$ by some $S$-independent polynomials involving multiples of Chern classes of $S$ on the diagonals of $S^h$. The diagonals are isomorphic to $S^{h'}$ for $h'\leq h$. We conclude that (\ref{eq1}) is expressible in terms of $S$-independent polynomials of classes on $S^{h'}$ of the form
\begin{equation*}
\prod^{h'}_{\ell=1}\pi^*_\ell\Big(\prod_{i \in I'_\ell, j\in J'_\ell} \gamma_{i,j} \cdot \mathrm{c}_1(S)^{k'_{1,\ell}} \cdot \mathrm{c}_2(S)^{k'_{2,\ell}} \Big),
\end{equation*}
where $\pi_\ell \colon S^{h'} \rightarrow S$ is now the projection to the $i$-th factor. By integrating these classes, we obtain that the formula in Theorem \ref{wall} is a $S$-independent polynomial expression in the characteristic numbers. 
\end{proof} 
\section{Genus 1 invariants}

\subsection{Evaluating the $I$-functions for genus $1$ invariants}	 
Recall that, in Nakajima basis, 
\[ 
D:=-\frac{1}{2}((2,1_S),(1,1_S)^{n-2}).
\]
We will use Theorem \ref{wall} to compute
\[
\langle D \rangle^{\Hilb_n(S)}_{1,d}. 
\]
The $I$-functions appearing in the wall-crossing formula for these invariants are of two types. If the part $\{2\}$ is present in a partition  $\mu^{j,i}$, then $\mu^{j,i}=D$, if not, then $\mu^{j,i}=1_{\Hilb_{n_i}(\BC^2)}$. Hence, the $I$-functions take the following form,  
\[
I_{n,1,d}(D,z), \quad  I_{n,1,d}(1,z). 
\]
In the next lemma, we evaluate them. 

\begin{lem} \label{evaluating}  
After making the substitution
\[
t_i \mapsto z+\alpha_i(S),
\]
and taking the Taylor expansion in $z$, we have the following identities in $H_A^*(S \times \Mbar_{1,1})[\![z^{\pm}]\!]$, 
\begin{align*}
I_{n,1,d}(D,z)&=t_1t_2\,\langle D \rangle_{1,d}^{\Hilb_n(\BC^2)} \cdot  [\mathrm{pt}]+t_1t_2\,\mathrm{Tr}_{n,d}\cdot [\Mbar_{1,1}], \\ 
I_{n,1,0}(1,z)&= t_1t_2\,  \langle 1 \rangle_{1,0}^{\Hilb_n(\BC^2)}\cdot [\mathrm{pt}]+ t_1t_2\, p(n)\cdot[\Mbar_{1,1}],
\end{align*} 
where $\mathrm{Tr}_{n,d}$ is the (non-reduced) trace of the quantum multiplication by $D$, and $p(n)$ is the number of partitions of $n$. For $d\neq 0$, 
\[
I_{n,1,d}(1,z)=0. 
\]
\end{lem}
\begin{proof} 
Since 
\[
H_*(\Mbar_{1,1})=\BQ[\mathrm{pt}] \oplus  \BQ [\Mbar_{1,1}], 
\]
to evaluate the $I$-function, we need to compute the genus 1 Gromov--Witten invariants of $\Hilb_n(\BC^2)$ with extra insertions given by $q^*[\mathrm{pt}]^\vee=q^*[\Mbar_{1,1}]$ and $q^*[\Mbar_{1,1}]^\vee=q^*[\mathrm{pt}]$, where $q \colon \Mbar_{g,k}(\Hilb_n(\BC^2),d) \rightarrow \Mbar_{g,k}$ is the natural projection. The former insertion corresponds to 
\[
\langle D \rangle_{1,d}^{\Hilb_n(\BC^2)},
\] 
while the latter gives the genus $1$ invariant of a fixed elliptic curve 
\[
\langle D \rangle_{(E,0_E),d}^{\Hilb_n(\BC^2)},
\] 
which in turn is equal to the trace of the quantum multiplication by the divisor $D$. This establishes the first identity.  
 
By the same reasoning, the second $I$-function for $d=0$ can be evaluated as follows,   
\[
I_{n,1,0}(1,z)= t_1t_2\langle 1 \rangle_{1,0}^{\Hilb_n(\BC^2)} \cdot [\mathrm{pt}]+ t_1t_2p(n)\cdot[\Mbar_{1,1}],
\]
where we used that 
\[ 
\langle 1 \rangle_{(E,0_E),0}^{\Hilb_n(\BC^2)}=e(\Hilb_n(\BC^2) )=p(n).  
\]
For $d\neq0$, the invariants with the identity insertion $1$ vanish by the string equation. 
\end{proof} 

To evaluate the wall-crossing formula, we need to determine the $t_i$-dependence of the Gromov--Witten invariants of $\Hilb_n(\mathbb{C}^2)$ that appear in the $I$-functions. By \cite{ILPT}, it is very simple (and independent of $n$), 
\begin{align} \label{dependence}
\begin{split}
\langle D \rangle_{1,d}^{\Hilb_n(\BC^2)}&=\frac{(t_1+t_2)^2}{t_1t_2} \cdot A_1\\
\mathrm{Tr}_{n,d}&=(t_1+t_2)\cdot A_2\\
\langle 1 \rangle_{1,0}^{\Hilb_n(\BC^2)}&=\frac{t_1+t_2}{t_1t_2} \cdot A_3\\
p(n)&=A_4,
\end{split}
\end{align}
where $A_i$ are rational numbers.

Define the generating series of $I$-functions, 
\begin{align*} 
I_{1}(Q,q)&:=\sum_{n\geq 1}\Big(\sum_{d\geq0} I_{n,1,d}(D,z)q^d\Big)Q^n\\
I_{2}(Q)&:=\sum_{n\geq 1} I_{n,1,0}(1,z)Q^n.
\end{align*}
By Lemma \ref{evaluating} and (\ref{dependence}), we have
\begin{align*} 
I_{1}(Q,q,z)&=A_1(Q,q)(2z+\mathrm{c}_1(S))^2 \boxtimes [\mathrm{pt}]  
+A_2(Q,q) (2z+\mathrm{c}_1(S))(z^2+z\mathrm{c}_1(S)+\mathrm{c}_2(S)) \boxtimes [\Mbar_{1,1}],\\
I_{2}(Q,z)&=A_3(Q) (2z+\mathrm{c}_1(S)) \boxtimes [\mathrm{pt}]
+A_4(Q)(z^2+z\mathrm{c}_1(S)+\mathrm{c}_2(S)) \boxtimes[\Mbar_{1,1}],
\end{align*}
which, by \cite{ILPT}, admit the following expressions 
\begin{align*}
A_1(Q,q)&=-\frac{1}{24} P(Q)\cdot (1+\log P(Q))\cdot T(Q,q), &A_2=T(Q,q),\\
A_3(Q)&=-\frac{1}{24}  P(Q) \cdot \log P(Q),  &A_4(Q)=P(Q)-1,
\end{align*}
such that 
\begin{equation*} 
P(Q):=\sum_{n\geq 0} p(n) Q^n, \quad 
T(Q,q):=\sum_{n\geq 0} \Big( \sum_{d\geq 0}\mathrm{Tr}_{n,d} q^d   \Big) Q^n,
\end{equation*}
as in (\ref{eqn:P_and_T}).
 
Observe that in the $I$-functions we do not consider the $n=0$ terms, since in Theorem~\ref{wall} the parts $n_i$ must be strictly positive. All the series above have a vanishing $Q^0$ term except for $P(Q)$ (hence the subtraction of $1$).
We are now ready to plug this into the formula from Theorem \ref{wall}. 
\subsection{Evaluating the wall-crossing formula} 
As in (\ref{eqn:g1_gen_func}), consider the generating series of the genus $1$ invariants with the insertion $D$, 
\begin{align*}
F(Q,q):=\sum_{n\geq 0} \Big( \sum_{d\geq 0} \langle D \rangle_{1,d}^{\Hilb_n(S)} q^d\Big) Q^n.
\end{align*}
Since there is only one insertion $D$, the $I$-function $I_{1}(Q,q,z)$ can be inserted only once. By choosing to insert $I_1(Q,q,z)$ via the first marking on $\mathsf{FM}_h(S)$, Theorem \ref{wall} gives the following formula on the level of generating series, 
\begin{equation} \label{formula}
F(Q,q) = \sum_{h\geq 0} \Big\langle I_1(Q,q,-\Psi_1) \cdot I_2(Q,-\Psi_2) \cdot \hdots \cdot I_2(Q,-\Psi_{h+1}) \Big\rangle_{1,1}^{\mathsf{FM}_{h+1}(S)}\big/h!.
\end{equation}
Moreover, since 
\[ 
[\mathrm{pt}] \cdot [\mathrm{pt}]=0
\]
in $H^*(\Mbar_{1,1})$, at most one summand with the class $[\mathrm{pt}]$ in either $I_1(Q,q)$ or $I_2(Q)$ can contribute. Hence, by integrating out the classes on the $\Mbar_{1,1}$-factor, the sum (\ref{formula}) splits as follows,
\begin{equation} \label{formula2}
\begin{split}
&\sum_{h\geq 0} A_1(Q,q)A_4(Q)^h\Big\langle (-2\Psi_1+\mathrm{c}_1(S))^2 \cdot  \prod^{h+1}_{i=2} \widetilde{\Psi}_i  \Big\rangle^{\mathsf{FM}_{h+1}(S)}\big/h!  \\
+&\sum_{h\geq 0} A_2(Q,q)A_3(Q)A_4(Q)^{h}\Big\langle  (-2\Psi_1+\mathrm{c}_1(S)) \widetilde{\Psi}_1 \cdot (-2\Psi_2+\mathrm{c}_1(S))\cdot  \prod^{h+2}_{i=3} \widetilde{\Psi}_i \Big\rangle^{\mathsf{FM}_{h+2}(S)}\big/h!,
\end{split}
\end{equation}
where the integrals now take place on $\mathsf{FM}_{h}(S)$ (without the $\Mbar_{1,1}$-factor), and
\[ 
\widetilde{\Psi}_i:=\Psi_i^2-\mathrm{c}_1(S)\Psi_i+\mathrm{c}_2(S).
\]
It remains to evaluate the Fulton--MacPherson integrals in the formula above.

\begin{prop} 
We have 
\begin{align*} 
\Big\langle (-2\Psi_1+\mathrm{c}_1(S))^2 \cdot  \prod^{h+1}_{i=2} \widetilde{\Psi}_i  \Big\rangle^{\mathsf{FM}_{h+1}(S)}\big/h! &=\int_S\mathrm{c}_1(S)^2 \binom{\int_S\mathrm{c}_2(S)-1}{h}, \\
\Big\langle  (-2\Psi_1+\mathrm{c}_1(S)) \widetilde{\Psi}_1 \cdot (-2\Psi_2+\mathrm{c}_1(S)) 
\cdot  \prod^{h+2}_{i=3} \widetilde{\Psi}_i \Big\rangle^{\mathsf{FM}_{h+2}(S)}\big/h!&= \\
\Big(\int_S\mathrm{c}_2(S)\mathrm{c}_1(S)\cdot \int_S\mathrm{c}_1(S)-&\int_S\mathrm{c}_1(S)^2\Big)\binom{\int_S\mathrm{c}_2(S)-2}{h}.
\end{align*}
\end{prop}
\begin{proof}
The higher-dimensional dilaton equation from \cite{N2} implies 
\begin{equation} \label{FMintegral}
\Big\langle (-2\Psi_1+\mathrm{c}_1(S))^2 \cdot  \prod^{h+1}_{i=2} \widetilde{\Psi}_i  \Big\rangle^{\mathsf{FM}_{h+1}(S)}\big/ h!= \Big(\int_S\mathrm{c}_2(S)-h\Big)	\Big\langle (-2\Psi_1+\mathrm{c}_1(S))^2 \cdot  \prod^{h}_{i=2} \widetilde{\Psi}_i  \Big\rangle^{\mathsf{FM}_{h}(S)}\big/h!.
\end{equation}
By induction, we obtain 
\[ 
(\ref{FMintegral})=\Big\langle (-2\Psi_1+\mathrm{c}_1(S))^2  \Big\rangle^{\mathsf{FM}_{1}(S) }  \binom{\mathrm{c}_2(S)-1}{h}.
\]
The class $\Psi_1$ vanishes on $\mathsf{FM}_{1}(S)=S$, hence, 
\[ 
(\ref{FMintegral})=\int_S\mathrm{c}_1(S)^2 \cdot  \binom{\int_S\mathrm{c}_2(S)-1}{h}. 
\]
	 
Similarly, by inductively applying the dilaton equation, we have  
\begin{multline*}
\Big\langle  (-2\Psi_1+\mathrm{c}_1(S)) \widetilde{\Psi}_1 \cdot (-2\Psi_2+\mathrm{c}_1(S)) 
	 \cdot  \prod^{h+2}_{i=3} \widetilde{\Psi}_i \Big\rangle^{\mathsf{FM}_{h+2}(S)}\big/h! \\
=\Big\langle  (-2\Psi_1+\mathrm{c}_1(S)) \widetilde{\Psi}_1 \cdot (-2\Psi_2+\mathrm{c}_1(S))  \Big\rangle^{\mathsf{FM}_2(S)}\binom{\int_S\mathrm{c}_2(S)-2}{h} .
\end{multline*}
For two points, we have $\mathsf{FM}_2(S)=\mathrm{Bl}_\Delta(S\times S)$, such that $\Psi_1=\Psi_2$ is the class of the exceptional divisor. Using the the blow-up relation 
\[ 
\Psi_1^3-\Psi_1^2\mathrm{c}_1(S)+\Psi_1\mathrm{c}_2(S)=0,
\]
we obtain 
\begin{align*}
&\int_{\mathrm{Bl}_\Delta(S\times S)}(-2\Psi_1+\pi_1^*\mathrm{c}_1(S))\cdot (-2\Psi_1+\pi_2^*\mathrm{c}_1(S)) 
\cdot \widetilde{\Psi}_1 )\\	
=&\int_{\mathrm{Bl}_\Delta(S\times S)}\pi_1^* \mathrm{c}_1(S) \cdot \pi_2^*\mathrm{c}_1(S) \cdot  \widetilde{\Psi}_1  
=\Big(\int_S\mathrm{c}_2(S)\mathrm{c}_1(S)\cdot \int_S\mathrm{c}_1(S)-\int_S\mathrm{c}_1(S)^2\Big),
\end{align*}
where we used the fact that $\pi_1=\pi_2$ over the diagonal $\Delta \subset S\times S$. 
\end{proof}

\begin{proof}[Proof of Theorem \ref{genus1}] By plugging the values of the Fulton--MacPherson integrals into formula~(\ref{formula2}), we obtain
\begin{multline*} 
(\ref{formula2}) = \sum_{h \geq 0} A_1(Q,q) A_4(Q)^h  \int_S\mathrm{c}_1(S)^2 \cdot \binom{\int_S\mathrm{c}_2(S)-1}{h} \\
+ \sum_{h \geq 0} A_2(Q,q) A_3(Q) A_4(Q)^h \Big(\int_S\mathrm{c}_2(S)\mathrm{c}_1(S)\cdot \int_S\mathrm{c}_1(S)-\int_S\mathrm{c}_1(S)^2\Big)\binom{\int_S\mathrm{c}_2(S)-2}{h},
\end{multline*}
which, by the binomial formula, simplifies to
\begin{multline*}
\int_S\mathrm{c}_1(S)^2 \cdot  A_1(Q,q)(A_4(Q)+1)^{\int_S\mathrm{c}_2(S)-1} \\
+ \Big(\int_S\mathrm{c}_2(S)\mathrm{c}_1(S)\cdot \int_S\mathrm{c}_1(S)-\int_S\mathrm{c}_1(S)^2\Big)  A_3(Q) A_2(Q,q)(A_4(Q)+1)^{\int_S\mathrm{c}_2(S)-2}.
\end{multline*}
Substituting the values of the generating series, we find that the first summand is equal to
\[
-\frac{1}{24} \int_S\mathrm{c}_1(S)^2\, P(Q)^{\int_S\mathrm{c}_2(S)-1}\, (1+\log P(Q))\, T(Q,q),
\]
while the second is equal to
\[
-\frac{1}{24} \Big(\int_S\mathrm{c}_2(S)\mathrm{c}_1(S)\cdot \int_S\mathrm{c}_1(S)-\int_S\mathrm{c}_1(S)^2\Big) P(Q)^{\int_S\mathrm{c}_2(S)-1} \log P(Q)\, T(Q,q).
\]
Summing them, we obtain
\[
-\frac{1}{24} P(Q)^{\int_S\mathrm{c}_2(S)-1}  T(Q,q)\Big( \int_S \mathrm{c}_1(S)^2+\Big(\int_S\mathrm{c}_2(S)\mathrm{c}_1(S)\cdot \int_S\mathrm{c}_1(S)\Big)\log P(Q)\Big),
\]
as desired.
\end{proof}

\end{document}